\UseRawInputEncoding
\documentclass[a4paper, 12pt,reqno]{amsart}
\usepackage{amsmath,amssymb,amsthm,enumerate}
\allowdisplaybreaks
\theoremstyle{plain}
\newtheorem{theorem}{Theorem}
\newtheorem*{theorem*}{Theorem}
\newtheorem{lemma}{Lemma}
\newtheorem*{lemma*}{Lemma}
\theoremstyle{definition}

\newtheorem*{definition*}{Definition}
\theoremstyle{remark}
\newtheorem{remark}{Remark}
\newtheorem{example}{Example}
\newtheorem*{remark*}{Remark}

\newtheorem*{statement*}{Statement}

\begin{document}
\title[Certain functions, operators,  and related problems]{Systems of functional equations and generalizations of certain functions}

\author{Symon Serbenyuk}



\maketitle
\text{\emph{simon6@ukr.net}}\\
\text{\emph{ 45~Shchukina St., Vinnytsia,  21012, Ukraine}}
\begin{abstract}

The present article is devoted to generalized Salem functions, the generalized shift operator,  and certain related problems. 
A description of  further investigations of the author of this article  is given.
 These investigations (in terms of various representations of real numbers) include   generalized Salem functions and generalizations of the Gauss-Kuzmin problem.

\end{abstract}

\section{Introduction}

Nowadays it is well known that functional equations and systems of functional equations are widely used in mathematics and other sciences.   Modelling functions with complicated local structure  by systems of functional equations is a shining example of their applications in  function theory (\cite{Symon2019}). 

Note that a class of functions with complicated local structure consists of singular (for example, \cite{{Salem1943}, {Zamfirescu1981}, {Minkowski}, {S.Serbenyuk 2017}}),  nowhere monotonic \cite{Symon2017, Symon2019}, and nowhere differentiable functions  (for example, see \cite{{Bush1952}, {Serbenyuk-2016}}, etc.).

Now researchers are trying to find simpler examples of singular  functions. Interest in such functions is explained by their connection with modelling  real objects, processes, and phenomena (in physics, economics, technology, etc.) and with different areas of mathematics (for example, see~\cite{BK2000, ACFS2011, Kruppel2009, OSS1995, Sumi2009, Takayasu1984, TAS1993}).
A brief historical remark on singular functions is given in~\cite{ACFS2017}.

One of the simplest examples of singular functions was introduced by Salem.  In \cite{Salem1943}, Salem modeled the function 
$$
s(x)=s\left(\Delta^2 _{\alpha_1\alpha_2...\alpha_n...}\right)=\beta_{\alpha_1}+ \sum^{\infty} _{n=2} {\left(\beta_{\alpha_n}\prod^{n-1} _{i=1}{q_i}\right)}=y=\Delta^{Q_2} _{\alpha_1\alpha_2...\alpha_n...},
$$
where $q_0>0$, $q_1>0$, and $q_0+q_1=1$. This function is a singular function. However,  
generalizations of the Salem function can be non-differentiable functions or those that do not have a derivative on a certain set.

Note that many researches are devoted to the Salem function and its generalizations (for example, see \cite{ACFS2017, Kawamura2010, Symon2015, Symon2017, Symon2019} and references in these papers).

Describing the present investigations, a certain generalization of  the $q$-ary representation is considered and certain properties of the generalized shift operator defined in terms of some of these representations are studied. Also, several related further investigations  of the author of this paper are noted. 
The main attention\footnote{The present investigations were presented by the author  in the preprint \cite{preprint19} in 2019.}  is given to modelling some generalization of the Salem function by certain systems of functional equations and by using the generalized shift operator.

\section{Some generalizations  of $q$-ary expansions of real numbers }

Let us consider the following  representation    introduced  by G.~Cantor in~\cite{C1869} in 1869. 

Let $Q\equiv (q_k)$ be a fixed sequence of positive integers, $q_k>1$,  $\Theta_k$ be a sequence of the sets $\Theta_k\equiv\{0,1,\dots ,q_k-1\}$, and $i_k\in\Theta_k$. Then
\begin{equation}
\label{eq: series1}
[0,1]\ni x=\Delta^{Q} _{i_1i_2...i_n...}\equiv \frac{i_1}{q_1}+\frac{i_2}{q_1q_2}+\dots +\frac{i_n}{q_1q_2\dots q_n}+\dots.
\end{equation}
It is easy to see that the last expansion is the  $q$-ary expansion
\begin{equation}
\label{eq: q-series}
\frac{\alpha_1}{q}+\frac{\alpha_2}{q^2}+\dots+\frac{\alpha_n}{q^n}+\dots \equiv \Delta^q _{\alpha_1\alpha_2...\alpha_n...}
\end{equation}
of numbers  from the closed interval $[0,1]$ whenever the condition $q_k=q$ holds for all positive integers $k$. Here $q$ is a fixed positive integer, $q>1$, and $\alpha_n\in\{0,1,\dots , q-1\}$.

Let us note that certain numbers from $[0,1]$ have two different representations by series \eqref{eq: series1}, i.e., 
$$
\Delta^Q _{i_1i_2\ldots i_{m-1}i_m000\ldots}=\Delta^Q _{i_1i_2\ldots i_{m-1}[i_m-1][q_{m+1}-1][q_{m+2}-1]\ldots}=\sum^{m} _{k=1}{\frac{i_k}{q_1q_2\dots q_k}}.
$$
Such numbers are called \emph{$Q$-rational}. The other numbers in $[0,1]$ are called \emph{$Q$-irrational}.

Let $c_1,c_2,\dots, c_m$ be a fixed 
ordered tuple of integers such that $c_j\in\{0,1,\dots, q_j-~1\}$ for $j=\overline{1,m}$. 

\emph{A cylinder $\Delta^Q _{c_1c_2...c_m}$ of rank $m$ with base $c_1c_2\ldots c_m$} is the following set 
$$
\Delta^Q _{c_1c_2...c_m}\equiv\{x: x=\Delta^Q _{c_1c_2...c_m i_{m+1}i_{m+2}\ldots i_{m+k}\ldots}\}.
$$
That is,  any cylinder $\Delta^Q _{c_1c_2...c_m}$ is a closed interval of the form
$$
\left[\Delta^Q _{c_1c_2...c_m000}, \Delta^Q _{c_1c_2...c_m[q_{m+1}-1][q_{m+2}-1][q_{m+3}-1]...}\right].
$$

By analogy, in the case of representation \eqref{eq: q-series}, we get
$$
\Delta^q _{\alpha_1\alpha_2\ldots\alpha_{m-1}\alpha_m000\ldots}=\Delta^q _{\alpha_1\alpha_2\ldots \alpha_{m-1}[\alpha_m-1][q-1][q-1][q-1]...\ldots}=\sum^{m} _{k=1}{\frac{\alpha_k}{q^k}}.
$$
Also, an arbitrary cylinder $\Delta^{q} _{c_1c_2...c_m}$ is a closed interval of the form
$$
\left[\Delta^q _{c_1c_2...c_m000}, \Delta^q _{c_1c_2...c_m[q-1][q-1][q-1]...}\right].
$$

\section{Shift operators}

In this section, the shift operator is described and the generalized shift operator is studied for the cases of $q$-ary expansions and of expansions of numbers in series \eqref{eq: series1}.

 \emph{The shift operator $\sigma$ of expansion \eqref{eq:  series1}} is of the following form
$$
\sigma(x)=\sigma\left(\Delta^Q _{i_1i_2\ldots i_k\ldots}\right)=\sum^{\infty} _{k=2}{\frac{i_k}{q_2q_3\dots q_k}}=q_1\Delta^{Q} _{0i_2\ldots i_k\ldots}.
$$
It is easy to see that 
\begin{equation*}
\begin{split}
\sigma^n(x) &=\sigma^n\left(\Delta^Q _{i_1i_2\ldots i_k\ldots}\right)\\
& =\sum^{\infty} _{k=n+1}{\frac{i_k}{q_{n+1}q_{n+2}\dots q_k}}=q_1\dots q_n\Delta^{Q} _{\underbrace{0\ldots 0}_{n}i_{n+1}i_{n+2}\ldots}.
\end{split}
\end{equation*}
Therefore, 
\begin{equation}
\label{eq: Cantor series 3}
x=\sum^{n} _{k=1}{\frac{i_k}{q_1q_2\dots q_k}}+\frac{1}{q_1q_2\dots q_n}\sigma^n(x).
\end{equation}

Note that 
$$
\sigma^n\left(\Delta^q _{\alpha_1\alpha_2\ldots \alpha_k\ldots}\right) =\sum^{\infty} _{k=n+1}{\frac{\alpha_k}{q^{k-n}}}=\Delta^{q} _{\alpha_{n+1}\alpha_{n+2}\ldots}.
$$

In \cite{S. Serbenyuk alternating Cantor series 2013} (see also \cite{{slides2013}, {preprint2013}}), the notion of  the generalized shift operator was introduced in terms of series \begin{equation}
\label{eq: alternating series1}
x=\Delta^{-Q} _{i_1i_2...i_n...}\equiv\frac{i_1}{-q_1}+\frac{i_2}{(-q_1)(-q_2)}+\dots +\frac{i_n}{(-q_1)(-q_2)\dots (-q_n)}+\dots .
\end{equation}
That is,   
$$
\sigma_m\left(\sum^{\infty} _{k=1}{\frac{(-1)^ki_k}{q_1q_2\cdots q_k}}\right)
$$
$$
=-\frac{i_1}{q_1}+\frac{i_2}{q_1q_2}-\frac{i_3}{q_1q_2q_3}+\dots + \frac{(-1)^{m-1}i_{m-1}}{q_1q_2\cdots q_{m-1}}+\frac{(-1)^{m}i_{m+1}}{q_1q_2\cdots q_{m-1}q_{m+1}}+\frac{(-1)^{m+1}i_{m+2}}{q_1q_2\cdots q_{m-1}q_{m+1}q_{m+2}}+\dots .
$$
The idea includes the following:  any number from a certain interval can be represented by two fixed sequences $(-q_n)$  and  $(i_n)$. The generalized shift operator maps the preimage into a number represented by the following two sequences $(-q_1,-q_2,\dots ,- q_{m-1}, -q_{m+1}, -q_{m+2}, \dots )$ and $(i_1,i_2,\dots , i_{m-1}, i_{m+1}, i_{m+2}, \dots )$. In terms of certain encodings of real numbers, this number can belong to another interval. 

Let us remark  that, in this section, the main attention is given to the generalized shift operator defined in terms of series \eqref{eq: series1} because this series is a generalization of a $q$-ary expansion and models (in the general case) a numeral system with a variable alphabet. Let us note that some numeral system is a numeral system with a variable alphabet whenever there exist at least two numbers     $k$ and $l$  such that the condition $A_k\ne A_l$  holds for  the representation $\Delta_{i_1i_2...i_n...}$  of numbers in terms of this numeral system, where  $i_k\in A_k$  and  $i_l\in A_l$, as well as  $k\ne l$.

Suppose a number $x\in [0,1]$ is  represented by  series \eqref{eq: series1}. Then  
$$
\sigma_m(x)=\sum^{m-1} _{k=1}{\frac{i_k}{q_1q_2\cdots q_k}}+\sum^{\infty} _{l=m+1}{\frac{i_l}{q_1q_2\cdots q_{m-1}q_{m+1}\cdots q_l}}.
$$
Denote by $\zeta_{m+1}$ the sum $\sum^{\infty} _{l=m+1}{\frac{i_l}{q_1q_2\cdots q_{m-1}q_{m+1}\cdots q_l}}$ and by $\vartheta_{m-1}$ the sum $\sum^{m-1} _{k=1}{\frac{i_k}{q_1q_2\cdots q_k}}$. Then $\zeta_{m+1}=q_m(x-\vartheta_m)$ and
\begin{equation}
\label{eq: generalized shift 1}
\sigma_m(x)=q_mx-(q_m-1)\vartheta_{m-1}-\frac{i_m}{q_1q_2\cdots q_{m-1}}.
\end{equation}

Let us remark that 
$$
\sigma(x)=\sigma_1(x)=\sum^{\infty} _{n=2}{\frac{i_n}{q_2q_3\cdots q_n}}=q_1\Delta^Q _{0i_2i_3...i_n...}
$$
and
$$
\sigma_m(x)=\Delta^Q _{i_1i_2...i_{m-1}000...}+q_m\Delta^Q _{\underbrace{0...0}_{m}i_{m+1}i_{m+2}...}=\Delta^Q _{i_1i_2...i_{m-1}0i_{m+1}i_{m+2}...}+(q_m-1)\Delta^Q _{\underbrace{0...0}_{m}i_{m+1}i_{m+2}...}.
$$

Let us consider some remarks on compositions of shift operators. 
\begin{remark}
Suppose $x=\Delta^Q _{i_1i_2...i_k...}$ and $m$ is a fixed positive integer. Then
$$
\sigma_{m}(x)=\sigma_{m}\left(\Delta^Q _{i_1i_2...i_k...}\right)=\Delta^{Q\setminus \{q_m\}} _{i_1i_2...i_{m-1}i_{m+1}...},
$$
$$
\sigma_{m}\circ \sigma_{m}(x)=\sigma^{2} _{m}(x)=\sigma^{2} _{m}\left(\Delta^Q _{i_1i_2...i_k...}\right)=\sigma_m\left(\sigma_m(\Delta^Q _{i_1i_2...i_k...})\right)=\sigma_{m}\left(\Delta^{Q\setminus \{q_m\}} _{i_1i_2...i_{m-1}i_{m+1}...}\right)=\Delta^{Q\setminus \{q_m, q_{m+1}\}} _{i_1i_2...i_{m-1}i_{m+2}...},
$$
and
$$
\underbrace{\sigma_{m}\circ \ldots \circ \sigma_{m}(x)}_{n}=\sigma^{n} _{m}\left(\Delta^Q _{i_1i_2...i_k...}\right)=\Delta^{Q\setminus \{q_m, q_{m+1}, \dots , q_{m+n-1}\}} _{i_1i_2...i_{m-1}i_{m+n-1}i_{m+n}...}.
$$
Here the sequence $Q\setminus \{q_m, q_{m+1}, \dots , q_{m+n-1}\}$ is $Q=(q_k)$ without the elements $q_m, q_{m+1}, \dots , q_{m+n-1}$. 

For the case of the shift operator, we get
\begin{equation*}
\begin{split}
\sigma^n(x) &=\sigma^n\left(\Delta^Q _{i_1i_2\ldots i_k\ldots}\right)=\Delta^{Q\setminus \{q_1, q_{2}, \dots , q_{n}\}} _{i_{n+1}i_{n+2}..i_{n+k}...}\\
& =q_1\dots q_n\Delta^{Q} _{\underbrace{0\ldots 0}_{n}i_{n+1}i_{n+2}\ldots}=\sum^{\infty} _{k=n+1}{\frac{i_k}{q_{n+1}q_{n+2}\dots q_k}}.
\end{split}
\end{equation*}
\end{remark}
\begin{remark}
Using the last remark, now let us describe a more general case. Suppose that $n_1$ and $n_2$ are two positive integers. Then
\begin{equation}
\label{eq: 2-composition}
\sigma_{n_2}\circ \sigma_{n_1}(x)=\sigma_{n_2}\left(\Delta^{Q\setminus \{q_{n_1}\}} _{i_1i_2...i_{n_1-1}i_{n_1+1}...}\right)=\begin{cases}
\Delta^{Q\setminus \{q_{n_1}, q_{n_2}\}} _{i_1i_2...i_{n_2-1}i_{n_2+1}...i_{n_1-1}i_{n_1+1}...}&\text{if $n_1>n_2$}\\
\Delta^{Q\setminus \{q_{n_1}, q_{n_2+1}\}} _{i_1i_2...i_{n_1-1}i_{n_1+1}...i_{n_2-1}i_{n_2}i_{n_2+2}...}&\text{if $n_1<n_2$}\\
\Delta^{Q\setminus \{q_{n_0}, q_{n_0+1}\}} _{i_1i_2...i_{n_0-1}i_{n_0+2}...}&\text{if $n_1=n_2=n_0$.}
\end{cases}
\end{equation}

For example,
$$
\sigma_5 \circ \sigma_2 (x)=\sigma_5 \circ \sigma_2\left(\Delta^Q _{i_1i_2\ldots i_k\ldots}\right)=\sigma_5 \left(\Delta^{Q\setminus \{q_2\}} _{i_1i_3i_4i_5\ldots i_k\ldots}\right)=\Delta^{Q\setminus \{q_2, q_6\}} _{i_1i_3i_4i_5i_7i_8i_9\ldots},
$$
$$
\sigma_3 \circ \sigma_6 (x)=\sigma_3 \circ \sigma_6\left(\Delta^Q _{i_1i_2\ldots i_k\ldots}\right)=\sigma_3 \left(\Delta^{Q\setminus \{q_6\}} _{i_1i_2i_3i_4i_5i_7i_8i_9\ldots}\right)=\Delta^{Q\setminus \{q_3, q_6\}} _{i_1i_2i_4i_5i_7i_8i_9\ldots},
$$
and
$$
\sigma_3 \circ \sigma_3 (x)=\sigma^2 _3(x)=\sigma_3 \circ \sigma_3\left(\Delta^Q _{i_1i_2\ldots i_k\ldots}\right)=\sigma_3 \left(\Delta^{Q\setminus \{q_3\}} _{i_1i_2i_4i_5i_6i_7i_8i_9\ldots}\right)=\Delta^{Q\setminus \{q_3, q_4\}} _{i_1i_2i_5i_6i_7i_8i_9\ldots}.
$$
\end{remark}
\begin{remark}
Let us note that, in the case of $q$-ary expansions of real numbers,  properties of the generalized shift operator are similar to those  of the generalized shift operator for expansions~\eqref{eq: series1}. Really, 
\begin{equation}
\label{eq: generalized shift q-ary}
\sigma_m\left(\Delta^q _{\alpha_1\alpha_2...\alpha_n...}\right)=qx-\frac{\alpha_m}{q^{m-1}}-(q-1)\sum^{m-1} _{k=1}{\frac{\alpha_k}{q^k}}.
\end{equation}
However, 
\begin{equation}
\label{eq: gshq}
\sigma_m\left(\Delta^q _{\alpha_1\alpha_2...\alpha_n...}\right)=\Delta^q _{\alpha_1\alpha_2...\alpha_{m-1}\alpha_{m+1}...}.
\end{equation}
\end{remark}

Let us note some auxiliary property which is useful for modelling the main object of this research. 
\begin{remark}
\label{rm: the main remark}
Suppose that numbers $x\in [0,1]$ are represented in terms of the $q$-ary representation and we need to delete the digits $\alpha_{n_1}, \alpha_{n_2}, \dots , \alpha_{n_k}$ (according to this fixed order) by using a composition of the generalized shift operators in $x=\Delta^q _{\alpha_1\alpha_2...\alpha_k...}$. Here $(n_k)$ is a finite fixed sequence of positive integers such that $n_i\ne n_j$ for $i\ne j$. That is, we model
$$
x_0=\Delta^q _{\alpha_1\alpha_2...\alpha_{n_1-1}\alpha_{n_1+1}...\alpha_{n_2-1}\alpha_{n_2+1}...\alpha_{n_k-1}\alpha_{n_k+1}\alpha_{n_k+2}...\alpha_{n_k+t}...},~~~\mbox{where}~t=1,2,3, \dots .
$$

Using \eqref{eq: 2-composition} and \eqref{eq: gshq}, a certain sequence $(\bar n_k)$ of positive integers is generated. That is, for all $i=\overline{1,k}$, 
$$
\bar n_i= n_i-\varrho_i,
$$
where $\varrho_i$ is the number of all numbers which are less than $n_i$ in the finite sequence $n_1, n_2, \dots , n_i$.

So,
$$
x_0=\sigma_{\bar n_k}\circ \sigma_{\bar n_{k-1}} \circ \dots \circ \sigma_{\bar n_2} \circ \sigma_{\bar n_1}(x).
$$

Let us consider the following example. Suppose $(n_k)=(1, 5, 7, 3, 6)$. That is, we generate a certain  sequence of the generalized shift operators for obtaining $x_0=\Delta^q _{\alpha_2\alpha_4\alpha_8\alpha_9\alpha_{10}...}$. Really, 
$$
\bar n_1=n_1=1: ~~~~~\sigma_{\bar n_1}(x)=\sigma_{n_1}(x)=\sigma_1(x)=\Delta^q _{\alpha_2\alpha_3\alpha_4\alpha_5\alpha_6\alpha_7\alpha_8\alpha_9\alpha_{10}...},
$$
$$
\bar n_2=n_2-1=5-1=4: 
$$
$$
\sigma_{\bar n_2} \circ\sigma_{\bar n_1}(x)=\sigma_4\circ \sigma_{1}(x)=\sigma_4(\sigma_{1}(x))=\sigma_4\left(\Delta^q _{\alpha_2\alpha_3\alpha_4\alpha_5\alpha_6\alpha_7\alpha_8\alpha_9\alpha_{10}...}\right)=\Delta^q _{\alpha_2\alpha_3\alpha_4\alpha_6\alpha_7\alpha_8\alpha_9\alpha_{10}...},
$$
$$
\bar n_3=n_3-2=7-2=5: 
$$
$$
\sigma_{\bar n_3} \circ\sigma_{\bar n_2} \circ\sigma_{\bar n_1}(x)=\sigma_5\circ \sigma_4\circ \sigma_{1}(x)=\Delta^q _{\alpha_2\alpha_3\alpha_4\alpha_6\alpha_8\alpha_9\alpha_{10}...}
$$
(in the third step, the digits $\alpha_1, \alpha_5, \alpha_7$ are deleted),
$$
\bar n_4=n_4-1=3-1=2:
$$
$$
\sigma_{\bar n_4} \circ\sigma_{\bar n_3} \circ\sigma_{\bar n_2} \circ\sigma_{\bar n_1}(x)=\sigma_2\circ\sigma_5\circ \sigma_4\circ \sigma_{1}(x)=\sigma_2\left(\Delta^q _{\alpha_2\alpha_3\alpha_4\alpha_6\alpha_8\alpha_9\alpha_{10}...}\right)=\Delta^q _{\alpha_2\alpha_4\alpha_6\alpha_8\alpha_9\alpha_{10}...},
$$
and finally, 
$$
\bar n_5=n_5-3=6-3=3:
$$
$$
\sigma_{\bar n_5} \circ\sigma_{\bar n_4} \circ\sigma_{\bar n_3} \circ\sigma_{\bar n_2} \circ\sigma_{\bar n_1}(x)=\sigma_3\circ\sigma_2\circ\sigma_5\circ \sigma_4\circ \sigma_{1}(x)=\sigma_3\left(\Delta^q _{\alpha_2\alpha_4\alpha_6\alpha_8\alpha_9\alpha_{10}...}\right)=\Delta^q _{\alpha_2\alpha_4\alpha_8\alpha_9\alpha_{10}...}.
$$
\end{remark}

\begin{lemma}
In the case of expansion \eqref{eq: series1}, the generalized shift operator has the following properties:
\begin{itemize}
\item $\sigma \circ \sigma^{m} _{2}(x)=\sigma^{m+1} (x)$.
\item  Suppose $(k_n)$  is a sequence of positive integers such that  $k_n=k_{n-1}+1$, $n=2,3, \dots$. Then 
$$
\sigma^{k_n}\circ \sigma_{k_n}\circ \sigma_{k_{n-1}}\circ \ldots \circ \sigma_{k_1}(x)=\sigma^{k_n+n}(x).
$$
\item  Suppose $(k_n)$  is an arbitrary finite subsequence of positive integers such that $k_1>k_2>\dots >k_n$. Then 
$$
\sigma^{k_n-n}\circ \sigma_{k_n}\circ \sigma_{k_{n-1}}\circ \ldots \circ \sigma_{k_1}(x)=\sigma^{k_n}(x).
$$
\item The mapping $\sigma_m$ is continuous at each point of the interval $\left(\inf\Delta^Q _{c_1c_2...c_m}, \sup\Delta^Q _{c_1c_2...c_m}\right)$. The endpoints of $\Delta^Q _{c_1c_2...c_m}$ are  points of discontinuity of the mapping.
\item The mapping $\sigma_m$ has a derivative almost everywhere (with respect to
the Lebesgue measure). If the mapping  has a derivative at the point $x=\Delta^Q _{i_1i_2...i_k...}$, then $\left(\sigma_m\right)^{'}=q_m$.
\item
$$
x-\sigma_m(x)=\frac{i_m}{q_1q_2\cdots q_m}+\frac{\sigma^m(x)}{q_1q_2\cdots q_m}(1-q_m).
$$
\end{itemize}
\end{lemma}
\begin{proof}
  All the properties follow from the definition of  $\sigma_m$  and equality \eqref{eq: generalized shift 1}.  

Let us consider a cylinder $\Delta^Q _{c_1c_2...c_n}$.  It is easy to see that $\lim_{x\to x_0}{\sigma_m(x)}=x_0$ holds for  any Q-irrational point  from $\Delta^Q _{c_1c_2...c_n}$ and all Q-rational points whenever   $m\ne n$. If  $m=n$, then 
$$
\lim_{x\to x_0+0}{\sigma_m(x)}=\sigma_m(x^{(1)} _0)=\sigma_m\left(\Delta^Q _{i_1i_2...i_{n-1}i_n000...}\right), 
$$
$$
\lim_{x\to x_0-0}{\sigma_m(x)}=\sigma_m(x^{(2)} _0)=\sigma_m\left(\Delta^Q _{i_1i_2...i_{n-1}[i_n-1][q_{n+1}-1][q_{n+2}-1]...}\right),
$$
and
$$
\sigma_m(x^{(1)} _0)-\sigma_m(x^{(2)} _0)=-\frac{1}{q_1q_2\cdots q_{m-1}}.
$$
In addition, 
$$
x-\sigma_m(x)=\vartheta_m+\frac{\sigma^m (x)}{q_1q_2\cdots q_m}-\vartheta_{m-1}-\zeta_{m+1}=\frac{i_m}{q_1q_2\cdots q_m}+\frac{\sigma^m(x)}{q_1q_2\cdots q_m}(1-q_m).
$$
\end{proof}

Let  us consider expansion \eqref{eq: alternating series1}. In this case,
$$
\sigma_m(x)=\sigma_m\left(\sum^{\infty} _{n=1}{\frac{(-1)^ni_n}{q_1q_2\cdots q_n}}\right)=\sum^{m-1} _{k=1}{\frac{(-1)^ki_k}{q_1q_2\cdots q_k}}+\sum^{\infty} _{j=m+1}{\frac{(-1)^{j-1}i_j}{q_1q_2\cdots q_{m-1}q_{m+1}\cdots q_j}}
$$
$$
=-q_mx+(1+q_m)\sum^{m-1} _{k=1}{\frac{(-1)^ki_k}{q_1q_2\cdots q_k}}+\frac{(-1)^mi_m}{q_1q_2\cdots q_{m-1}}.
$$
So, $\sigma_m$ is a piecewise linear function since  $\sum^{m} _{k=1}{\frac{(-1)^kc_k}{q_1q_2\cdots q_k}}$ is constant for the set 
$\Delta^Q _{c_1c_2...c_m}$.

In the paper \cite{preprint2019},   the generalized shift operator is investigated in more detail.  
In the next articles of the author of this paper, the notion of the generalized shift operator will be investigated in more detail and applied by the author of the present article  in terms of various representations of real numbers (e.g., positive and alternating Cantor series and their generalizations, as well as Luroth, Engel series, etc., various continued fractions).

Let us consider certain applications of the generalized shift operator. One can model generalizations of the Gauss-Kuzmin problem and generalizations of the Salem function.

\section{Generalizations of the Gauss-Kuzmin problem}

This problem is one of the first and still one of the most important results in the metrical theory of continued
fractions \cite{Lasku}.  The problem was formulated by Gauss and  the first solution was received by Kuzmin \cite{Kuzmin}. The problem is investigated by a number of researchers for different types  of continued fractions 
(for example, see \cite{{Kalpazidou}, {Lasku},  {Lasku2014}} and references in the papers).

The Gauss-Kuzmin problem is to calculate the limit
$$\lim_{n\to\infty}{\lambda\left(E_n(x)\right)},
$$
where $\lambda(\cdot)$ is the Lebesgue measure of a set and  the set $E_n(x)$ is a set of the form
$$
E_n=\left\{z: \sigma^n (z)<x\right\}.
$$
Here $z=\Delta _{i_1i_2...i_k...}$, i.e., $\Delta _{i_1i_2...i_k...}$ is a certain representation of real numbers, $\sigma$ is the shift operator. 

Generalizations of the Gauss-Kuzmin problem are to calculate the limit
$$
\lim_{k\to\infty}{\lambda\left(\tilde E^{} _{n_k}(x)\right)},
$$
for  sets of the following forms:
\begin{itemize}
\item 
$$
\tilde E^{} _{n_k}(x)=\left\{z: \sigma_{n_k}\circ \sigma_{n_{k-1}}\circ \ldots \circ \sigma_{n_1}(z)<x\right\}
$$
including (here $(n_k)$ is a certain fixed sequence of positive integers) the cases when $(n_k)$ is a constant sequence. 

\item the set $\tilde E^{} _{n_k}(x)$ under the condition that $n_k=\psi(k)$, where $\psi$ is a certain function of the positive integer argument. 

\item
$$
\tilde E^{} _{n_k}(x)=\left\{z: \underbrace{\sigma_{n_k}\circ \sigma_{n_{k-1}}\circ \ldots \circ \sigma_{n_1}}_{\varphi(m,k,c)}(z)<x\right\},
$$
where $\varphi$ is a certain function and  $m,c$ are some parameters (if applicable). That is, for example, 
$$
\tilde E^{} _{n_k}(x)=\left\{z: \underbrace{\sigma_{m}\circ \sigma_{m}\circ \ldots \circ \sigma_{m}}_{k}(z)<x\right\},
$$
where $k>c$ and $c$ is a fixed positive integer, 
or
$$
\tilde E^{} _{n_k}(x)=\left\{z: \underbrace{\sigma_{m}\circ \sigma_{m}\circ \ldots \circ \sigma_{m}}_{k}(z)<x\right\},
$$
where $k \equiv 1 (\mod c) $ and $c>1$ is a fixed positive integer. 

\item In the general case,
$$
\tilde E^{} _{n_k}(x)=\left\{z: \underbrace{\sigma_{\psi(\varphi(m,k,c))}\circ  \ldots \circ \sigma_{\psi(1)}}_{\varphi(m,k,c)}(z)<x\right\}.
$$
\end{itemize}

In addition, one can formulate  such problems in terms of the shift operator. For example, one can formulate the Gauss-Kuzmin problem for  the following sets:
$$
\tilde E^{} _{n_k}(z)=\left\{z: \sigma^{n_k}(z)<\sigma^{k_0}(z)\right\},
$$
where $k_0$, $(n_k)$ are a fixed number and a fixed sequence.
$$
\tilde E^{} _{n_k}(x)=\left\{z: \sigma^{n_k}(z)<\sigma^{k_0}(x)\right\}.
$$
$$
\tilde E^{} _{n}(x)=\left\{z: \sigma^{\psi(n)}(z)<x\right\},
$$
where $\psi(n)$ is a certain function of the positive integer argument.

In addition, 
$$
\tilde E^{} _{n}(z)=\left\{z: \sigma^{\psi(n)}(z)<\sigma^{\varphi(n)}(z)\right\},
$$
$$
\tilde E^{} _{n}(x)=\left\{z: \sigma^{\psi(n)}(z)<\sigma^{\varphi(n)}(x)\right\},
$$
where $\psi, \varphi$  are certain functions of the positive integer arguments.

It is easy to see that similar problems can be formulated for the case of the generalized shift operator.

In next articles of the author of the present article, such problems will be considered by the author of this article in terms of various  numeral systems (with a finite or infinite alphabet, with a constant or variable alphabet, positive, alternating, and sign-variable expansions, etc.).


\section{Generalizations of the Salem function}

Let us consider certain functions whose arguments are represented by the $q$-ary expansion.

Suppose $(n_k)$ is a fixed sequence of positive integers such that $n_i\ne n_j$ for $i\ne j$ and such that  for any $n\in\mathbb N$ there exists a number $k_0$ for which the condition $n_{k_0}=n$ holds. Suppose $\bar n_k=n_k-\varrho_k$ for all $k=1,2, 3, \dots$, where $\varrho_k$ is the number of all numbers which are less than $n_k$ in the finite sequence $n_1, n_2, \dots , n_k$.

\begin{theorem}
Let $P_q=\{p_0,p_1,\dots , p_{q-1}\}$ be a fixed tuple of real numbers such that $p_i\in (-1,1)$, where $i=\overline{0,q-1}$, $\sum_i {p_i}=1$, and $0=\beta_0<\beta_i=\sum^{i-1} _{j=0}{p_j}<1$ for all $i\ne 0$. Then the following system of functional equations
\begin{equation}
\label{eq: system-q}
f\left(\sigma_{\bar n_{k-1}}\circ \sigma_{\bar n_{k-2}}\circ \ldots \circ \sigma_{\bar n_1}(x)\right)=\beta_{\alpha_{n_k}}+p_{\alpha_{n_k},}f\left(\sigma_{\bar n_{k}}\circ \sigma_{\bar n_{k-1}}\circ \ldots \circ \sigma_{\bar n_1}(x)\right),
\end{equation}
where $x=\Delta^q _{\alpha_1\alpha_2...\alpha_k...}$, $k=1,2, \dots$, and $\sigma_0(x)=x$, has the unique solution
$$
g(x)=\beta_{\alpha_{n_1}}+\sum^{\infty} _{k=2}{\left(\beta_{\alpha_{n_k}}\prod^{k-1} _{j=1}{p_{\alpha_{n_j}}}\right)}
$$
in the class of determined and bounded on $[0, 1]$ functions. 
\end{theorem}
\begin{proof}
Since the function $g$ is a determined on $[0,1]$ function, using system~\eqref{eq: system-q},  we get 
$$
g(x)=\beta_{\alpha_{n_1}}+p_{\alpha_{n_1}}g(\sigma_{\bar n_1}(x))
$$
$$
=\beta_{\alpha_{n_1}}+p_{\alpha_{n_1}}(\beta_{\alpha_{n_2}}+p_{\alpha_{n_2}}g(\sigma_{\bar n_2}\circ\sigma_{\bar n_1}(x)))=\dots
$$
$$
\dots =\beta_{\alpha_{n_1}}+\beta_{\alpha_{n_2}}p_{\alpha_{n_1}}+\beta_{\alpha_{n_3}}p_{\alpha_{n_1}}p_{\alpha_{n_2}}+\dots +\beta_{\alpha_{n_k}}\prod^{k-1} _{j=1}{p_{\alpha_{n_j}}}+\left(\prod^{k} _{t=1}{p_{\alpha_{n_t}}}\right)g(\sigma_{\bar n_k}\circ \dots \circ \sigma_{\bar n_2}\circ \sigma_{\bar n_1}(x)).
$$

So,
$$
g(x)=\beta_{\alpha_{n_1}}+\sum^{\infty} _{k=2}{\left(\beta_{\alpha_{n_k}}\prod^{k-1} _{j=1}{p_{\alpha_{n_j}}}\right)}
$$
since $g$ is a  determined and bounded on $[0,1]$ function and 
$$
\lim_{k\to\infty}{g(\sigma_{\bar n_k}\circ \dots \circ \sigma_{\bar n_2}\circ \sigma_{\bar n_1}(x))\prod^{k} _{t=1}{p_{\alpha_{n_t}}}}=0,
$$
where
$$
\prod^{k} _{t=1}{p_{\alpha_{n_t}}}\le \left( \max_{0\le i\le q-1}{p_i}\right)^k\to 0, ~~~ k\to \infty.
$$
\end{proof}

\begin{example}
Suppose 
$$
(n_k)=(1, 5, 7, 3, 6, 10, 2, 4, 8, 9, 11, 12, 13, 14, 15, \dots ).
$$
Then, using arguments described in Remark~\ref{rm: the main remark}, we get the following:  $\bar n_1=n_1=1$,
$$
\bar n_2=n_2-1=5-1=4, \ \ \ \ \ \ \ \  \bar n_3=n_3-2=7-2=5, \ \ \ \  \ \ \ \ \bar n_4=n_4-1=3-1=2,
$$
$$
\bar n_5=n_5-3=3, \ \ \  \ \ \ \  \ \bar n_6=n_6-5=5, \ \ \ \ \ \ \  \ \ \ \ \ \ \ \ \ \ \  \  \bar n_7=n_7-1=1,
$$
$$
\bar n_8=n_8-3=1, \ \ \  \ \ \ \  \ \bar n_9=n_9-7=1, \ \ \ \ \ \ \  \ \ \ \ \ \ \ \ \ \ \ \ \  \  \bar n_{10}=n_{10}-8=1,
$$
$\bar n_{11}=\bar n_{12}=\dots = \bar n_{10+k}=n_{10+k}- (n_{10+k}-1)=1$ for $k=1, 2, 3, \dots $.

So, we obtain the function
$$
g(x)=\beta_{\alpha_{n_1}}+\sum^{\infty} _{k=2}{\left(\beta_{\alpha_{n_k}}\prod^{k-1} _{j=1}{p_{\alpha_{n_j}}}\right)}
$$
according to the sequence $(n_k)=(1, 5, 7, 3, 6, 10, 2, 4, 8, 9, 11, 12, 13, 14, 15, \dots )$. That is, for $x=\Delta^q _{\alpha_1\alpha_2...\alpha_k...}$, we have
$$
y=g(x)=\beta_{\alpha_1(x)}+\beta_{\alpha_5(x)}p_{\alpha_1(x)}+\beta_{\alpha_7(x)}p_{\alpha_1(x)}p_{\alpha_5(x)}+\beta_{\alpha_3(x)}p_{\alpha_1(x)}p_{\alpha_5(x)}p_{\alpha_7(x)}+\dots .
$$
\end{example}

\begin{theorem} The following properties hold: 
\begin{itemize}
\item The function $g$ is continuous at any $q$-irrational point of $[0,1]$.
\item The function $g$ is continuous at the $q$-rational point
$$
x_0=\Delta^{q} _{\alpha_1\alpha_2...\alpha_{m-1}\alpha_m 000...}=\Delta^{q} _{\alpha_1\alpha_2...\alpha_{m-1}[\alpha_m-1] [q-1][q-1][q-1]...}
$$
whenever a sequence $(n_k)$ is such that the conditions  $k_0=\max\{k:  n_k \in \{1,2,\dots, m\}\}$, $n_{k_0}=m$, and $n_1, n_2, \dots , n_{k_0-1}\in \{1, 2, \dots , m-1\}$  hold. Otherwise, the $q$-rational  point  $x_0$ is a point of discontinuity.
\item The set of all points of discontinuity of the function $g$ is a countable, finite, or empty set. It  depends on the sequence $(n_k)$.
\end{itemize}
 
\end{theorem}
\begin{proof}
Let us note that a certain fixed function $g$ is given by a fixed sequence $(n_k)$ described above. One can write our  mapping by the following:
$$
g: x=\Delta^q _{\alpha_1\alpha_2...\alpha_k...}\to ~\beta_{\alpha_{n_1}}+\sum^{\infty} _{k=2}{\left(\beta_{\alpha_{n_k}}\prod^{k-1} _{l=1}{p_{\alpha_{n_l}}}\right)}=\Delta^{g(x)} _{\alpha_{n_1}\alpha_{n_2}...\alpha_{n_k}...}=g(x)=y.
$$

Let $x_0=\Delta^q _{\alpha_1\alpha_2...\alpha_k...}$ be an arbitrary $q$-irrational number from $[0,1]$. Let $x=\Delta^q _{\gamma_1\gamma_2...\gamma_k...}$ be a  $q$-irrational number such that the condition $\gamma_{n_j}=\alpha_{n_j}$ holds  for all $j=\overline{1,k_0}$, where $k_0$ is a certain positive integer. That is, 
$$
x=\Delta^q _{\gamma_1...\gamma_{n_1-1}\alpha_{n_1}\gamma_{n_1+1}...\gamma_{n_2-1}\alpha_{n_2}...\gamma_{(n_{(k_0-1)}+1)}...\gamma_{(n_{k_0}-1)}\alpha_{n_{k_0}}\gamma_{n_{k_0}+1}...\gamma_{n_{k_0}+k}...}, ~k=1,2,\dots .
$$
Then
$$
g(x_0)=\Delta^{g(x)} _{\alpha_{n_1}\alpha_{n_2}...\alpha_{n_{k_0}}\alpha_{n_{k_0+1}}...},
$$
$$
g(x)=\Delta^{g(x)} _{\alpha_{n_1}\alpha_{n_2}...\alpha_{n_{k_0}}\gamma_{n_{k_0+1}}...\gamma_{n_{k_0}+k}...}.
$$
Since $g$ is a bounded function, $ g(x) \le 1$, we have $g(x)-g(x_0)=$
$$
=\left(\prod^{k_0} _{j=1}{p_{\alpha_{n_j}}}\right) \left(\beta_{\gamma_{n_{k_0+1}}}+\sum^{\infty} _{t=2}{\left(\beta_{\gamma_{n_{k_0+t}}}\prod^{k_0+t-1} _{r=k_0+1}{p_{\gamma_{n_r}}}\right)}-\beta_{\alpha_{n_{k_0+1}}}-\sum^{\infty} _{t=2}{\left(\beta_{\alpha_{n_{k_0+t}}}\prod^{k_0+t-1} _{r=k_0+1}{p_{\alpha_{n_r}}}\right)}\right)
$$
$$
=\left(\prod^{k_0} _{j=1}{p_{\alpha_{n_j}}}\right)\left(g(\sigma_{\bar n_{k_0}}\circ\ldots \sigma_{\bar n_2} \circ \sigma_{\bar n_1}(x))-g(\sigma_{\bar n_{k_0}}\circ\ldots \sigma_{\bar n_2} \circ \sigma_{\bar n_1}(x_0))\right),
$$
and
$$
|g(x)-g(x_0)|\le \delta\prod^{k_0} _{j=1}{p_{\alpha_{n_j}}}\le \delta\left(\max\{p_0,\dots , p_{q-1}\}\right)^{k_0}\to 0 ~~~~~~~(k_0\to\infty). 
$$
Here $\delta$ is a certain real number.

So, $\lim_{x\to x_0}{g(x)}=g(x_0)$, i.e., the function $g$ is continuous at any $q$-irrational point. 

Let $x_0$ be a $q$-rational number, i.e.,
$$
x_0=x^{(1)} _0=\Delta^{q} _{\alpha_1\alpha_2...\alpha_{m-1}\alpha_m 000...}=\Delta^{q} _{\alpha_1\alpha_2...\alpha_{m-1}[\alpha_m-1] [q-1][q-1][q-1]...}=x^{(2)} _0.
$$
Then there exist positive integers $k^{*}$ and $k_0$ such that
$$
y_1=g\left(x^{(1)} _0\right)=\Delta^{g(x)} _{\alpha_{n_1}\alpha_{n_2}...\alpha_{n_{k^{*}}}...\alpha_{n_{k_0}}000...},
$$
$$
y_2=g\left(x^{(2)} _0\right)=\Delta^{g(x)} _{\alpha_{n_1}\alpha_{n_2}...\alpha_{n_{k^{*}-1}}[\alpha_{n_{k^{*}}}-1]\alpha_{n_{k^{*}+1}}...\alpha_{n_{k_0}}[q-1][q-1][q-1]...}.
$$
 Here $n_{k^{*}}=m$, $n_{k^{*}}\le n_{k_0}$, and $k_0$ is a number such that $\alpha_{n_{k_0}}\in\{\alpha_1, \dots, \alpha_{m-1}, \alpha_m\}$ and ${k_0}$ is the maximum position of any number from  $\{1,2,\dots , m\}$ in the sequence $(n_k)$.

Let us consider the fact~(Section~2 in \cite{Symon2019} and  the paper~\cite{Salem1943}, since such an expansion of numbers is an  analytic representation of the Salem function) that  the representation $\Delta^{g(x)} _{\alpha_{1}\alpha_{2}...\alpha_{k}...}$ is the following  whenever the conditions $(n_k)=(k)$ and $p_j>0$ for all $j=\overline{0,q-1}$, where $k=1,2, \dots $, hold:
\begin{equation}
\label{eq: Pq}
[0,1]\ni x=\Delta^{P_q} _{\alpha_{1}\alpha_{2}...\alpha_{k}...}=\beta_{\alpha_1}+\sum^{\infty} _{k=2}{\left(\beta_{\alpha_1}\prod^{k-1} _{l=1}{p_{\alpha_l}}\right)}
\end{equation}
This representation is the $q$-ary representation whenever the condition
$$
0< p_0=p_1=\dots=p_{q-1}=\frac{1}{q}
$$
holds. Also, certain numbers have two different  such representations, and the rest of the numbers have a unique such representation.  That is, 
$$
z_1=\Delta^{P_q} _{\alpha_1\alpha_2...\alpha_{m-1}\alpha_m 000...}=\Delta^{P_q} _{\alpha_1\alpha_2...\alpha_{m-1}[\alpha_m-1] [q-1][q-1][q-1]...}=z_2.
$$
Really, 
$$
z_1-z_2=\beta_{\alpha_1}+\sum^{m} _{k=2}{\left(\beta_{\alpha_1}\prod^{k-1} _{l=1}{p_{\alpha_l}}\right)}-\beta_{\alpha_1}-\sum^{m-1} _{k=2}{\left(\beta_{\alpha_1}\prod^{k-1} _{l=1}{p_{\alpha_l}}\right)}-\beta_{\alpha_m-1}\prod^{m-1} _{l=1}{p_{\alpha_l}}-p_{\alpha_m-1}\prod^{m-1} _{l=1}{p_{\alpha_l}}=0.
$$
Since the Salem function is a strictly increasing function, conditions for having $x_1<x_2$ or $x_1>x_2$ are identical in terms of  the $q$-ary representation and of representation~\eqref{eq: Pq}. 

Using the case of a $q$-ary irrational number, let us consider the limits
$$
\lim_{x\to x_0+0}{g(x)}=\lim_{x\to x^{(1)} _0}{g(x)}=g(x^{(1)} _0)=y_1,~~~\lim_{x\to x_0-0}{g(x)}=\lim_{x\to x^{(2)} _0}{g(x)}=g(x^{(2)} _0)=y_2.
$$

Whence $y_1=y_2$ whenever a sequence $(n_k)$ is such that the conditions $n_{k_0}=m$, $k_0=\max\{k:  n_k \in \{1,2,\dots, m\}\}$,  and $n_1, n_2, \dots , n_{k_0-1}\in \{1, 2, \dots , m-1\}$ hold.

So,  the set of all points of discontinuity of the function $g$ is a countable, finite, or empty set. It  depends on the sequence $(n_k)$.
\end{proof}

Suppose $(n_k)$ is a fixed sequence and $c_{n_1}, c_{n_2}, \dots , c_{n_r}$ is a fixed tuple of numbers $c_{n_j}\in\{0,1,\dots , q-1\}$, where $j=\overline{1,r}$ and $r$ is a fixed positive integer.

Let us consider the following set
$$
\mathbb S_{q, (c_{n_r})}\equiv \left\{x: x=\Delta^q _{\alpha_1\alpha_2...\alpha_{n_1-1}\overline{c_{n_1}}\alpha_{n_1+1}...\alpha_{n_2-1}\overline{c_{n_2}}...\alpha_{n_{r}-1}\overline{c_{n_r}}\alpha_{n_r+1}...\alpha_{n_r+k}...}\right\},
$$
where $k=1,2,\dots $, and $\overline{c_{n_j}}\in \{c_{n_1}, c_{n_2}, \dots , c_{n_r}\}$ for all $j=\overline{1,r}$. This set has non-zero Lebesgue measure (for example, similar sets are investigated in  terms of other representations of numbers in~\cite{S. Serbenyuk alternating Cantor series 2013}). It is easy to see that 
$\mathbb S_{q, (c_{n_r})}$ maps to
$$
g\left(\mathbb S_{q, (c_{n_r})}\right)\equiv\left\{y: y=\Delta^{g(x)} _{c_{n_1} c_{n_2}\dots  c_{n_r}\alpha_{n_{r+1}}...\alpha_{n_{r+k}}...}\right\}
$$
under $g$.

For a  value $\mu_g \left(\mathbb S_{q, (c_{n_r})}\right)$ of the increment, the following is true.
$$
\mu_g \left(\mathbb S_{q, (c_{n_r})}\right)=g\left(\sup\mathbb S_{q, (c_{n_r})}\right)-g\left(\inf\mathbb S_{q, (c_{n_r})}\right)=\Delta^{g(x)} _{c_{n_1} c_{n_2}\dots  c_{n_r}[q-1][q-1][q-1]...}-\Delta^{g(x)} _{c_{n_1} c_{n_2}\dots  c_{n_r}000...}=\prod^{r} _{j=1}{p_{c_{n_j}}}.
$$
Let us note that one can consider the intervals $\left[\inf\mathbb S_{q, (c_{n_r})}, \sup\mathbb S_{q, (c_{n_r})}\right]$. Then $\sup\mathbb S_{q, (c_{n_r})}-\inf\mathbb S_{q, (c_{n_r})}=$
$$
=\Delta^q _{\underbrace{[q-1][q-1]...[q-1]}_{n_1-1}\overline{c_{n_1}}\underbrace{[q-1][q-1]...[q-1]}_{n_2-1}\overline{c_{n_2}}...\underbrace{[q-1][q-1]...[q-1]}_{n_r-1}\overline{c_{n_r}}(q-1)}
$$
$$
-\Delta^q _{\underbrace{00...0}_{n_1-1}\overline{c_{n_1}}\underbrace{00...0}_{n_2-1}\overline{c_{n_2}}...\underbrace{00...0}_{n_r-1}\overline{c_{n_r}}(0)}=1-\sum^{r} _{j=1}{\frac{q-1}{q^{n_j}}}
$$
and
\begin{equation}
\label{eq: increment}
\mu_g \left(\mathbb S_{q, (c_{n_r})}\right)=\mu_g \left(\left[\inf\mathbb S_{q, (c_{n_r})}, \sup\mathbb S_{q, (c_{n_r})}\right]\right)=\prod^{r} _{j=1}{p_{c_{n_j}}}.
\end{equation}

Let us remark that the function $g$ is the Salem function defined in terms of the $q$-ary representation whenever $(n_k)=(k)$ (i.e., $(\bar n_k)=const=1$) and all $p_i (i=\overline{0,q-1})$ are positive or non-negative numbers.

So, one can formulate the following statements.
\begin{theorem}
The function $g$ has the following properties:
\begin{enumerate}
\item If $p_j\ge 0$ or $p_j>0$ for all $j=\overline{0,q-1}$, then:
\begin{itemize}
\item $g$ does not have intervals of monotonicity on $[0,1]$ whenever the condition $n_k=k$ holds for no  more than a finite number of values of $k$; 
\item $g$ has at least one interval of monotonicity on $[0,1]$ whenever  the condition $n_k\ne k$ holds for  a finite number of values of $k$; 
\item $g$ is a monotonic non-decreasing function (in the case when $p_j\ge 0$ for all $j=\overline{0,q-1}$) or is a strictly increasing function (in the case when $p_j> 0$ for all $j=\overline{0,q-1}$) whenever the condition $n_k=k$ holds for  $k\in\mathbb N$.
\end{itemize}
\item If there exists   $p_j=0$, where $j=\overline{0,q-1}$, then $g$ is  constant almost everywhere on $[0,1]$.
\item If there exists  $p_j<0$ (other $p_j$ are positive), where $j=\overline{0,q-1}$, and the condition $n_k=k$ holds for  almost all $k\in\mathbb N$, then $g$ does not have intervals of monotonicity on $[0,1]$.
\end{enumerate}
\end{theorem}

Let us note that  the last statements follow from~\eqref{eq: increment}.

Let us consider a cylinder $\Delta^q _{c_1c_2...c_n}$. We obtain 
$\mu_g \left(\Delta^q _{c_1c_2...c_n}\right)=$
$$
=\Delta^{g(x)} _{\underbrace{[q-1][q-1]...[q-1]}_{e_1}\overline{c_{1}}\underbrace{[q-1][q-1]...[q-1]}_{e_2}\overline{c_{2}}...\underbrace{[q-1][q-1]...[q-1]}_{e_n}\overline{c_{n}}(q-1)}
$$
$$
-\Delta^{g(x)} _{\underbrace{00...0}_{e_1}\overline{c_{1}}\underbrace{00...0}_{e_2}\overline{c_{2}}...\underbrace{00...0}_{e_n}\overline{c_{n}}(0)},
$$
where $\overline{c_{j}}\in\{c_1,c_2,\dots , c_n\}$, $j=\overline{1,n}$, and $(e_n)$ is a certain sequence of numbers from $\mathbb N \cup\{0\}$.

So, differential properties of $g$ depend on a sequence $(n_k)$ and the set of numbers $P_q=\{p_0,p_1,\dots , p_{q-1}\}$.
\begin{statement*}
The function $g$ can be a singular or non-differentiable function. It  depends on the sequence $(n_k)$ and $P_q=\{p_0,p_1,\dots , p_{q-1}\}$.
\end{statement*}

Differential properties including special partial cases will be considered in the next articles of the author of this paper since the  technique of proofs  introduced by Salem in~\cite{Salem1943} is not suitable for proving statements in our general  case. 

In addition, let us note the following.
\begin{lemma}
Let $\eta$ be a random variable  defined by the following form 
$$
\eta=  \Delta^{q} _{\xi_{1}\xi_{2}...\xi_{k}...},
$$
where
$\xi_k=\alpha_{n_k}$, $k=1,2,3,\dots $, and the digits $\xi_{k}$ are  random and take the values $0,1,\dots ,q-1$ with probabilities ${p}_{0}, {p}_{1}, \dots , {p}_{q-1}$.
That is,  $\xi_n$ are independent and $P\{\xi_{k}=\alpha_{n_k}\}=p_{\alpha_{n_k}}$, $\alpha_{n_k}\in\{0,1,\dots q-1\}$. 
 Here $(n_k)$ is a sequence of positive integers such that $n_i\ne n_j$ for $i\ne j$ and such that  for any $n\in\mathbb N$ there exists a number $k_0$ for which the condition $n_{k_0}=n$ holds.

The distribution function ${F}_{\eta}$ of the random variable $\eta$ can be
represented by
$$
{F}_{\eta}(x)=\begin{cases}
0,&\text{ $x< 0$}\\
\beta_{\alpha_{n_1}(x)}+\sum^{\infty} _{k=2} {\left({\beta}_{\alpha_{n_k}(x)} \prod^{k-1} _{r=1} {{p}_{\alpha_{n_k}(x)}}\right)},&\text{ $0 \le x<1$}\\
1,&\text{ $x\ge 1$,}
\end{cases}
$$
where $x=\Delta^{q} _{\alpha_{n_1}\alpha_{n_2}...\alpha_{n_k}...}$.
\end{lemma}

A method of  the corresponding proof is described in~\cite{Symon2017}.

\begin{theorem}
The Lebesgue integral of the function $g$ can be calculated by the
formula
$$
\int^1 _0 {g(x)dx}=\frac{1}{q-1}\sum^{q-1} _{j=0}{\beta_j}.
$$
\end{theorem}
\begin{proof}

By $A$ denote the sum $\frac{1}{q}\sum^{q-1} _{j=0}{\beta_j}$ and by $B$ denote the sum $\frac{1}{q}\sum^{q-1} _{j=0}{p_j}$. Since equality~\eqref{eq: generalized shift q-ary} holds, we obtain
$$
x=\frac{1}{q}\sigma_m(x)+\frac{(q-1)}{q}\sum^{m-1} _{k=1}{\frac{\alpha_k}{q^k}}+\frac{\alpha_m}{q^{m}}
$$
and
$$
dx=\frac{1}{q}d(\sigma_m(x)).
$$

In the general case, for arbitrary positive integers $n_1$ and $n_2$, using equalities \eqref{eq: 2-composition} and \eqref{eq: gshq}, we have 
$$
\sigma_{n_2}\circ \sigma_{n_1}(x)=\sigma_{n_2}\left(\Delta^{q} _{\alpha_1\alpha_2...\alpha_{n_1-1}\alpha_{n_1+1}...}\right)=\begin{cases}
\Delta^{q} _{\alpha_1\alpha_2...\alpha_{n_2-1}\alpha_{n_2+1}...\alpha_{n_1-1}\alpha_{n_1+1}...}&\text{if $n_1>n_2$}\\
\Delta^{q} _{\alpha_1\alpha_2...\alpha_{n_1-1}\alpha_{n_1+1}...\alpha_{n_2-1}\alpha_{n_2}\alpha_{n_2+2}...}&\text{if $n_1<n_2$}\\
\Delta^{q} _{\alpha_1\alpha_2...\alpha_{n_0-1}\alpha_{n_0+2}\alpha_{n_0+3}...}&\text{if $n_1=n_2=n_0$.}
\end{cases}
$$
Also,
$$
\sigma_{n_1}(x)=\frac{1}{q}\sigma_{n_2}\circ\sigma_{n_1}(x)+\frac{q-1}{q}\sum^{n_2-1} _{k=1}{\frac{\alpha_k}{q^k}}+\frac{\alpha_{n_2}}{q^{n_2}}~~~\mbox{whenever}~~~n_1>n_2,
$$
$$
\sigma_{n_1}(x)=\frac{1}{q}\sigma^2 _{n_0}+\frac{q-1}{q}\sum^{n_0-1} _{k=1}{\frac{\alpha_k}{q^k}}+\frac{\alpha_{n_0+1}}{q^{n_0}}~~~\mbox{whenever}~~~n_1=n_2=n_0,
$$
and
$$
\sigma_{n_1}(x)=\frac{1}{q}\sigma_{n_2}\circ\sigma_{n_1}(x)+\frac{q-1}{q}\left(\sum^{n_1-1} _{k=1}{\frac{\alpha_k}{q^k}}+\frac{\alpha_{n_1+1}}{q^{n_1}}+\frac{\alpha_{n_1+2}}{q^{n_1+1}}+\dots + \frac{\alpha_{n_2}}{q^{n_2-1}}\right)+\frac{\alpha_{n_2+1}}{q^{n_2}}~~~\mbox{whenever}~~~n_1<n_2.
$$
Hence, for any posirive integer $k$, the following is true:
$$
d\left(\sigma_{n_{k-1}}\circ \dots \circ \sigma_{n_2}\circ \sigma_{n_1}(x)\right)=\frac{1}{q}d\left(\sigma_{n_k}\circ \sigma_{n_{k-1}}\circ \dots \circ \sigma_{n_2}\circ \sigma_{n_1}(x)\right),
$$
where $\sigma_0(x)=x$.

By analogy, we obtain
$$
\sigma_{\bar n_2}\circ \sigma_{\bar n_1}(x)=\begin{cases}
\Delta^{q} _{\alpha_1\alpha_2...\alpha_{n_2-1}\alpha_{n_2+1}...\alpha_{n_1-1}\alpha_{n_1+1}...}&\text{whenever $n_1>n_2$}\\
\Delta^{q} _{\alpha_1\alpha_2...\alpha_{n_1-1}\alpha_{n_1+1}...\alpha_{n_2-1}\alpha_{n_2+1}\alpha_{n_2+2}...}&\text{whenever  $n_1<n_2$}
\end{cases}
$$
and
$$
d\left(\sigma_{\bar n_{k-1}}\circ \dots \circ \sigma_{\bar n_2}\circ \sigma_{\bar n_1}(x)\right)=\frac{1}{q}d\left(\sigma_{\bar n_k}\circ \sigma_{\bar n_{k-1}}\circ \dots \circ \sigma_{\bar n_2}\circ \sigma_{\bar n_1}(x)\right), 
$$
where $k\in\mathbb N$ and $\sigma_0(x)=x$.

So, we have
$$
\int^1 _0 {g(x)dx}=\sum^{q-1} _{j=0}{\int^{\frac{j+1}{q}} _{\frac{j}{q}} {g(x)dx}}=\sum^{q-1} _{j=0}{\int^{\frac{j+1}{q}} _{\frac{j}{q}} {\left(\beta_j+p_jg(\sigma_{\bar n_1}(x))\right)dx}}
$$
$$
=\frac{1}{q}\sum^{q-1} _{j=0}{\beta_j}+\frac{1}{q}\left(\sum^{q-1} _{j=0}{p_j}\right)\int^1 _0 {g(\sigma_{\bar n_1}(x))d(\sigma_{\bar n_1}(x))}
$$
$$
=\frac{1}{q}\sum^{q-1} _{j=0}{\beta_j}+\frac{1}{q}\left(\sum^{q-1} _{j=0}{p_j}\right)\left(\sum^{q-1} _{j=0}{\int^{\frac{j+1}{q}} _{\frac{j}{q}} {\left(\beta_j+p_jg(\sigma_{\bar n_2}\circ \sigma_{\bar n_1}(x))\right)d(\sigma_{\bar n_1}(x))}}\right)
$$
$$
=A+B\left(A+B\int^1 _0 {g(\sigma_{\bar n_2}\circ \sigma_{\bar n_1}(x)))d(\sigma_{\bar n_2}\circ \sigma_{\bar n_1}(x))}\right)
$$
$$
=A+AB+B^2\left(\sum^{q-1} _{j=0}{\int^{\frac{j+1}{q}} _{\frac{j}{q}} {\left(\beta_j+p_jg(\sigma_{\bar n_3}\circ\sigma_{\bar n_2}\circ \sigma_{\bar n_1}(x))\right)d(\sigma_{\bar n_2}\circ\sigma_{\bar n_1}(x))}}\right)
$$
$$
=A+AB+B^2\left(A+B\int^1 _0 {g(\sigma_{\bar n_3}\circ\sigma_{\bar n_2}\circ \sigma_{\bar n_1}(x)))d(\sigma_{\bar n_3}\circ\sigma_{\bar n_2}\circ \sigma_{\bar n_1}(x))}\right)
$$
$$
=A+AB+AB^2+B^3\left(A+B\int^1 _0 {g(\sigma_{\bar n_4}\circ\sigma_{\bar n_3}\circ\sigma_{\bar n_2}\circ \sigma_{\bar n_1}(x)))d(\sigma_{\bar n_4}\circ\sigma_{\bar n_3}\circ\sigma_{\bar n_2}\circ \sigma_{\bar n_1}(x))}\right)=\dots
$$
$$
\dots = A+AB+\dots +AB^{k-1}+B^k\left(A+B\int^1 _0 {g(\sigma_{\bar n_{k+1}}\circ\sigma_{\bar n_k}\circ\ldots \circ \sigma_{\bar n_1}(x)))d(\sigma_{\bar n_{k+1}}\circ\sigma_{\bar n_k}\circ \ldots \circ \sigma_{\bar n_1}(x))}\right).
$$

Since
$$
B^{k+1}=\left(\frac{1}{q}\sum^{q-1} _{j=0}{p_j}\right)^{k+1}=\left(\frac{1}{q}\right)^{k+1} \to 0 ~\text{as}~ k\to\infty,
$$
we obtain
$$
\int^1 _0{g(x)dx}=\lim_{k\to\infty}{\left(\sum^{k} _{t=0}{AB^t}+B^{k+1}\int^1 _0 {g(\sigma_{\bar n_{k+1}}\circ\sigma_{\bar n_k}\circ\ldots \circ \sigma_{\bar n_1}(x)))d(\sigma_{\bar n_{k+1}}\circ\sigma_{\bar n_k}\circ \ldots \circ \sigma_{\bar n_1}(x))}\right)}
$$
$$
=\sum^{\infty} _{k=0}{AB^k}=\left(\sum^{q-1} _{j=0}{\beta_j}\right)\left(\sum^{\infty} _{k=0}{\frac{1}{q^{k+1}}}\right)=\frac{1}{q-1}\sum^{q-1} _{j=0}{\beta_j}.
$$
\end{proof}

In the next articles of the author of this paper, generalizations and properties of solutions of  system \eqref{eq: system-q} of functional equations will be investigated for the cases of various numeral systems (with a finite or infinite alphabet, with a constant or variable alphabet, positive, alternating, and sign-variable expansions, etc.). 

\begin{remark}
 It can   be interesting to consider the case when $(n_k)$ is an arbitrary fixed  sequence (finite or infinite) of positive integers. Then  the function $g$ can be a constant function, a linear function,  or a function having pathological (complicated) structure, etc. It depends on  $(n_k)$. Such problems will be investigated in the next papers of the author of this article.
\end{remark}

\end{document}